\documentclass[11pt]{article}
\usepackage[labelfont=bf]{caption}

\usepackage{amsmath,amsfonts,amssymb,amsthm,booktabs,color,epsfig,graphicx,hyperref,url}
\usepackage{enumitem}

\theoremstyle{plain}
\newtheorem{theorem}{Theorem}

\newtheorem{lemma}{Lemma}
\newtheorem{corollary}{Corollary}

\usepackage{natbib}
\setcitestyle{round}

\theoremstyle{definition}

\newtheorem{assumption}{Assumption}

\newcommand{\pr}{\prime}

\newcommand{\bs}{\boldsymbol{s}}

\newcommand{\convP}{\stackrel{P}{\to}}

\begin{document}

  \title{\bf Consistently recovering the signal from noisy functional data}
  \author{Siegfried H\"ormann\thanks{Corresponding author.
Email: shoermann@tugraz.at}\hspace{.2cm}\\
    Institute of Statistics, Graz University of Technology\\
    and \\
    Fatima Jammoul \\
    Institute of Statistics, Graz University of Technology}
  \maketitle

\begin{abstract}
In practice most functional data cannot be recorded on a continuum, but rather at discrete time points. It is also quite common that these measurements come with an additive error, which one would like eliminate for the statistical analysis. When the measurements for each functional datum are taken on the same grid,  the underlying signal-plus-noise model can be viewed as a factor model. The signals refer to the common components of the factor model, the noise is related to the idiosyncratic components. We formulate a framework which allows to consistently recover the signal by a PCA based factor model estimation scheme. Our theoretical results hold under rather mild conditions, in particular we don't require specific smoothness assumptions for the underlying curves and allow for a certain degree of autocorrelation in the noise.
\end{abstract}



\section{Introduction}

As a consequence of the explosive increase of available data in practically all areas of life, data science (as a generic term for the diverse data processing disciplines) has been steadily gaining popularity over the past decades. Handling massive data not only requests larger human and computational resources, but also demands and entails the development of new scientific approaches. Within the field of statistics, functional data analysis (FDA) is one such area which has experienced a massive surge in interest over the last couple of years. Its basic objective is to handle data, where each observation  $X_t$ can be viewed as a process  $X_t(s)$  defined on some continuous domain $\mathcal{U}$. To date a number of  comprehensive text books (e.g.,\ \cite{ramsaysilverman05, ferratyvieu:2006, horvath:kokoszka:2012, hsing:eubank:2015}) can be consulted for a general overview of the numerous tools and methods which have been developed. The most active domains of research arguably are related to functional regression and dimension reduction techniques (see e.g.,\ \cite{cuevas:2014} for a survey). The latest developments are compactly summarized in \cite{goiavieu:2016} and \cite{aneirosetal:2019}.

Our contribution concerns the preprocessing step which is needed in most functional data analyses. Typically $X_t(s)$ cannot be sampled over the entire continuum, but rather we record a discretized version.  Moreover, it is not uncommon that measuring devices become more error-prone with an increase in sampling frequency, so that often we actually observe a discretized and noisy version of $X_t(s)$. The question how to deal with the discrepancy between the theoretical model and the actual measurements is often the starting point in theoretical and practical  FDA problems. While existing literature has been focusing on different variants of smoothing techniques, we want to explore this problem now from a different perspective.

Throughout this paper we are dealing with a stationary sequence of functional data $(X_t\colon t\geq 1)$, where each data point is of the form $(X_t(s)\colon s\in[0,1])$. In the common case where $\mathcal{U}$ is an interval, the restriction to $[0,1]$ comes with no loss of generality. We assume that $X_t(s)$ is recorded at discrete time points $0\leq s_1< s_2< \cdots <s_p\leq 1$. We focus on data where these sampling points are identical for each $t$ and where the number of sampling points $p$ is large. Such a setting is very common when we have machine recorded data. Irrespective of whether we have a sparse or a dense observation grid, most papers in FDA literature assume that measurements might come with an additional error, yielding actual observations of the form
\begin{equation}\label{signoise}
y_t=(X_t(s_1),\ldots, X_t(s_p))^\top+(u_{t1},\ldots, u_{tp})^\top=:X_t(\bs)+u_t.
\end{equation}
Here and in the sequel $\bs=(s_1,\ldots, s_p)$ and $g(\bs)=(g(s_1),\ldots, g(s_p))^\top$.
When we want to process such data, we need to eliminate the noise $u_t$ in order to guarantee a valid inference for $X_t$. Most available approaches towards this end  assume $X_t(s)$ to be a smooth curve and then propose to estimate the latent signal with diverse smoothing techniques. The smoothing is mostly done curve-by-curve, for example via spline smoothing (e.g.\ \citet{ramsaysilverman05}), kernel smoothing (e.g.\ \citet{wandjones95}) or local linear regression (e.g. \citet{fan:1993}). Alternatively, \citet{Staniswalis:Lee:1998} suggest to smooth the empirical covariance of the raw data  in order to get an estimate of $\Gamma^X(s,s')=\mathrm{Cov}(X_t(s),X_t(s'))$. The eigenfunctions of the smoothed covariance are then used to get a proxy for the Karhunen-Lo\`eve expansion of $X_t$.

 In this paper we investigate an alternative method, which doesn't make use of the common smoothing techniques and which in turn doesn't rely on potentially unverifiable smoothness of the curves. Rather we use that data  $y_t$ given in \eqref{signoise} can be shown to follow some (approximate) factor model, where $X_t(\bs)$ und $u_t$ correspond to the common components and idiosyncratic components of $y_t$, respectively. This means that the components $X_t(\bs)$ are driven by some lower dimensional process (the common factors) and hence are strongly correlated, whereas the components of $u_t$ are uncorrelated or only mildly correlated and represent unsystematic fluctuations (see e.g.,\ \cite{mardia:kent:bibby:1979} for a basic introduction).
 
Recovering  $X_t(\bs)$ can thus be seen as estimating the common components in a factor model. In a companion paper \cite{hormann:jammoul:2020a} we have pointed and worked out the connection to factor models and thoroughly explored the practical performance of existing MLE or PCA based estimation schemes on real and simulated data. We found that these approaches lead to very convincing and nicely interpretable results. In a variety of considered settings the factor model estimates outperform competing smoothing techniques in terms of the resulting mean square approximation error. This advantage is particularly remarkable for moderately sized $p$ and large $T$ or when the underlying signal $X(s)$ is not smooth.
 
The primary purpose of this article is to provide the theoretical validation of the factor approach. Following \citet{fanetal2013} we will explore a principal components based estimator for the common components. In \cite{hormann:jammoul:2020a} we have experienced that it leads to slightly better practical performance when compared to a likelihood based approach (e.g.,\ \citet{baili2012}). Moreover, it is very simple to implement.
When $EX_t=0$, then the estimator is of the form 
\begin{equation}\label{e:pcaapproach}
(\hat X_1(\bs),\ldots, \hat X_T(\bs))=
Y\hat E\hat E^\top,
\end{equation}
where $Y=(y_1,\ldots, y_T)$ and where $\hat E=(\hat e_1,\ldots, \hat e_L)$ are the eigenvectors of $\frac{1}{T}Y^\top Y$ belonging to the~$L$ largest eigenvalues $\hat\gamma_1\geq \cdots\geq \hat\gamma_L$. In our theorems below we will prove that
$$
\sup_{1\leq t\leq T}\sup_{1\leq i\leq p}|\hat X_t(s_i)-X_t(s_i)|\convP 0.
$$
Here $\convP$ denotes convergence in probability. 
Our results are based on $T\to\infty$ and $p=p(T)\to\infty$. While our reasoning is similar to \cite{fanetal2013}, we are working with assumptions  specific for functional data. In particular, we are going to allow the number of factors $L$ to diverge, i.e.\ $L=L(T)\to\infty$.  Still, we will be able to work under much less restrictive moment and dependence assumptions compared to current factor model literature. They are also much less restrictive than assumptions in comparable FDA smoothing results. We will expand on this in more detail in Section~\ref{ss:comp}.

In the next section we formulate our assumptions and the main theorems. In Section~\ref{ss:proofTH} we give the core steps of the proofs. We conclude in Section \ref{s:conclusion}. Technical lemmas and detailed proofs are given in the Appendix.

\section{Results}\label{s:est}
 In this section we formulate and discuss our assumptions and present our large sample results.  It will be assumed throughout that $(X_t\colon t\geq 1)$ is stationary and that the curves  $(X_t(s)\colon s\in [0,1])$ are
 second-order stochastic processes. We then can define $\mu(s)=EX_t(s)$ and $\Gamma^X(s,s^\prime)=\mathrm{Cov}(X_t(s),X_t(s^\prime))$, respectively. We will be considering the setting where the sample size $T$ and the number of sampling points $p=p(T)$ tend to infinity.  The data are sampled as in \eqref{signoise}.
 
\subsection{Assumptions}\label{s:assumptions}
Let us now list the detailed assumptions on the processes $(u_t)$ and $(X_t)$. 

\begin{assumption}\label{a:noise}
The noise process $(u_t)$ is i.i.d.\ zero mean and independent of the signals~$(X_t)$. The processes $(u_{ti}\colon 1\leq i\leq p)$ are Gaussian with covariance function $\gamma^u(h) = \text{Cov}(u_{t(i+h)},u_{ti})$, such that
$
\sum_{h\in\mathbb{Z}}|\gamma^u(h)|\leq C_u<\infty.
$
\end{assumption}
 Assuming Gaussianity of the noise processes simplifies our proofs and assumptions and seems reasonably justified in the context of modelling measurement errors. 
It is important to note that we don't impose that $u_{ti}$ and $u_{tj}$ are uncorrelated, as it is often required for factor models. In many real data situations it is likely that errors of consecutive measurements are at least mildly correlated. Our main theorems below can be extended under suitable moment and dependence assumptions for the $(u_{ti}\colon 1\leq i\leq p)$. See \ref{appendix:gauss}.

Summability of the autocovariances implies that the correlation between the errors decays quickly with increasing lag. This is in line with the paradigm of approximate factor models. Such models have been first discussed in \cite{chamberlain1983} in context of macroeconomic applications and are now quite commonly used in high dimensional factor model settings. Several works have further investigated approximate factor models since then, see e.g.,\ \cite{bai2003, bailiao2012, choi2012, baili2016, baing:2019}.

For the signal process $(X_t)$ we will allow for temporal dependence. A concept that proved to be rather useful in this context, is the so-called $L^p$-$m$-approximability, which has been proposed in \citet{hoermann2010}.  A sequence is said to be $L^p$-$m$-approximable, if it has an ergodic representation of the form $X_t=g(\delta_t,\delta_{t-1},\ldots)$ with some measurable functional $g$ and i.i.d.\ elements $(\delta_t)$ in some measurable space and satisfies
\begin{equation}\label{e:Lpm}
\sum_{m\geq 1} (E\|X_t-g(\delta_t,\ldots,\delta_{t-m},\delta_{t-m-1}^\prime,\delta_{t-m-2}^\prime,\ldots)\|^p)^{1/p}<\infty.
\end{equation}
Here, $\|\cdot\|$ denotes the $L^2$-norm for square integrable functions on $[0,1]$ and $(\delta_{t}^\prime)$ is an independent copy of the innovations $(\delta_t)$. This means that we can approximate $X_t$ by coupled $m$-dependent sequences sufficiently well. Several popular functional time series models (e.g.\ functional ARMA or functional GARCH) fall into this framework. Unlike diverse mixing conditions,  $L^p$-$m$-approximability is often much easier to establish. For the purpose of illustration consider a functional AR(1) model  $X_t=\varrho(X_{t-1})+\delta_t$. Under suitable assumptions on the linear operator $\varrho$ (see e.g. \cite{bosq:2000}) we get by iterative application of the recursion that $X_t=\sum_{k\geq 0}\varrho^k(\delta_{t-k})$, and hence the required representation holds. A sufficient assumption is $\|\varrho\|<1$. Here the coupled version is $X_t^{(m)}=\sum_{k=0}^m\varrho^k(\delta_{t-k})+\sum_{k>m}\varrho^k(\delta_{t-k}^\prime)$ and thus if the $\delta_t$ have $p$ moments we get
$$
\sum_{m\geq 1} (E\|X_t-X_t^{(m)}\|^p)^{1/p}\leq (E\|\delta_0-\delta_0^\prime\|^p)^{1/p}\sum_{m\geq 1}\sum_{k>m}\|\varrho\|^k<\infty.
$$
This shows that with increasing $m$, the  $L^p$ error between the original data and the coupled $m$-dependent sequence goes to zero fast enough to guarantee summability in \eqref{e:Lpm}.
More details can be found in \cite{hoermann2010}.
\begin{assumption}\label{a:signal}
(i) The process $(X_t\colon t\geq 1)$ is zero mean and $L^4$-$m$-approximable. (ii) The curves $X_t=(X_t(s)\colon s\in [0,1])$ define fourth order random processes (i.e.\ $
\sup_{s\in[0,1]} EX_1^4(s)\leq C_X<\infty$)
with a continuous covariance kernel. (iii) It holds that $E\sup_{s\in [0,1]} X^2_1(s)<\infty$. (iv)~Observations $X_t$ lie in some $L$-dimensional function space, where $L=L(T)$ may diverge with $T\to\infty$.
\end{assumption}
The zero mean assumption is not necessary, but it simplifies the presentation. 
In practice $\mu(\bs)$ will be estimated via $\hat\mu(\bs)=T^{-1}(y_1+\cdots+y_T)$ and data will be centered. 
It is implicit in the definition of $L^4$-$m$-approximability that  the process $(X_t)$ is strictly stationary and ergodic, which is why the conditions (ii) and (iii) can be restricted to $X_1$. 
 The main theoretical restriction is Assumption~\ref{a:signal}~(iv). It is motivated by Mercer's theorem which yields under the assumption of a continuous covariance kernel that for a functional variable $X_t$ it holds that
 \begin{equation}\label{e:mercer}
\sup_{s\in [0,1]}E\left|X_t(s)-\mu(s)-\sum_{\ell= 1}^L x_{t\ell}\varphi_\ell(s)\right|^2\to 0,\quad L\to\infty,
\end{equation}
where  $\varphi_\ell$ denotes the eigenfunction of the covariance operator $\Gamma^X$ belonging to the $\ell$-th eigenvalue $\lambda_\ell$ and the $x_{t\ell}=\int_0^1(X_t(s)-\mu(s))\varphi_\ell(s) ds=:\langle X_t-\mu,\varphi_\ell\rangle$ are functional principal component scores. In most real data problems the convergence in \eqref{e:mercer} is very fast, so that even for relatively small $L$ the variables $X_t^L(s)=\mu(s)+\sum_{\ell =1}^L x_{t\ell}\varphi_\ell(s)$ define only a slight perturbations of $X_t(s)$, lying in the $L$-dimensional space spanned by the functional principal components $\varphi_1,\ldots,\varphi_L$. Hence this Assumption~\ref{a:signal}~(iv) holds asymptotically for a very broad class of functional data. 

Now the key observation that we are going to exploit in the context of this paper is that  under Assumptions~\ref{a:noise} and~\ref{a:signal} the $y_t$ given in \eqref{signoise} follow some $L$-factor model, where $X_t(\bs)$ is the common component and $u_t$ the idiosyncratic component. See \cite{hormann:jammoul:2020a}. 

\begin{assumption}\label{a:pcs}
For the eigenfunctions $\varphi_\ell$  it holds that $\max_{1\leq k,\ell\leq L}\left|p^{-1}\sum_{i=1}^p\varphi_k(s_i)\varphi_\ell(s_i)\right|=O(1)$ as $T\to\infty$.
\end{assumption}
At this stage we stress once more that $p=p(T)$ and $L=L(T)$, which explains the asymptotic statement for $T\to\infty$. 
If $s_i$ are dense in $[0,1]$ then $p^{-1}\sum_{i=1}^p\varphi_k(s_i)\varphi_\ell(s_i)\approx \int_0^1\varphi_k(s)\varphi_\ell(s)ds=\delta_{k,\ell}$ (Kronecker delta). Hence, this assumption is very mild.

\subsection{Theorems}\label{s:theorems}
We are now ready to formulate our asymptotic results. We distinguish the scenarios  where $L$ is fixed (Theorem~\ref{thm}) and where $L=L(T)\to\infty$ (Theorem~\ref{thm2}). Our main results are formulated with $L$ known. At the end of this section we comment on the estimation of~$L$. All following asymptotic formulations are understood as $T\to\infty$.

\begin{theorem}\label{thm}
Let Assumptions~\ref{a:noise}--\ref{a:pcs} hold. Assume that $p=p(T)$ diverges at a subexponential rate and let $L$ be arbitrary but fixed. Furthermore, suppose that $p/\hat\gamma_L=O_P(1)$. Then for any $t\leq T$
$$\max_{1 \leq j \leq p } | \hat{X}_t(s_j) - X_t(s_j)|= O_P\Big(\frac{1}{T^{1/4}}+\frac{T^{1/4}}{\sqrt{p}}\Big).$$
Hence the estimator is consistent if $\sqrt{T}/p\to 0$. Moreover, 
$$\max_{1 \leq j \leq p } \max_{1\leq t\leq T}| \hat{X}_t(s_j) - X_t(s_j)|= O_P\left(\max_{1\leq t\leq T}\|f_t\|\left(\frac{1}{T^{1/4}}+\frac{T^{1/4}}{\sqrt{p}}\right)\right),$$ where $f_t=(x_{t1}/\sqrt{\lambda_1},\ldots, x_{tL}/\sqrt{\lambda_L})^\top$.
\end{theorem}

Theorem~\ref{thm} needs further explanation. First, we comment on $p/\hat\gamma_L=O_P(1)$. The following lemma provides a simple condition when this assumption holds.
\begin{lemma}\label{l:lambdagamma} Assume that the sampling points $s_i$ are equidistant and that
\begin{equation}\label{e:inc}
\sup_{s\in [0,1]}E|X_t(s+h)-X_t(s)|^2=O(h),\quad h\to 0.
\end{equation}
Then, under the conditions of Theorem~\ref{thm}, we have $\max_{i\geq 1}|\lambda_i-\hat\gamma_i/p|=O_P\left(1/\sqrt{p}+1/\sqrt{T}\right)$.
\end{lemma}
The lemma shows, among others, that $p/\hat\gamma_L$ is a consistent estimator for $\lambda_L^{-1}$ and hence is $O_P(1)$. Condition \eqref{e:inc} holds if the covariance kernel $\Gamma^X$ is uniformly continuous. We would like to point out that the conditions in Lemma~\ref{l:lambdagamma} may be adapted and weakened. We decided to disentangle this problem from the theorem and to present conditions which are neat and not very technical.  

Concerning the uniform consistency in Theorem~\ref{thm} we need in addition a bound for $\max_{1\leq t\leq T}\|f_t\|$. Such bounds can easily be obtained under higher order moment conditions. We formulate the following lemma.
\begin{lemma}\label{l:fs}
Suppose that Assumption~\ref{a:signal} holds and that $L$ is fixed.\\[1ex]
\emph{(i)} If $E\|X_1\|^{q}<\infty$ for some $q>0$, then
$\max_{1\leq t\leq T}\|f_t\|=o_P\left(T^{1/q^\pr}\right)$ for all $0<q^\pr<q$.\\
\emph{(ii)} If  $E\exp(\alpha \|X_1\|)<\infty$ for some $\alpha>0$, then 
$\max_{1\leq t\leq T}\|f_t\|= O_P \left(\log(T)\right).$
\end{lemma}

The lemma implies that if we request moments of order $>4$ for the signal process $X_t$, then our Theorem~\ref{thm} provides uniform consistency over time and cross-section.

Next we consider the setting where the dimension of the function space is allowed to diverge. Not surprisingly, this requires extra conditions, e.g.\ on the eigenvalues of $\Gamma^X$ and on the growth rate of $L$. Instead of listing several different sets of conditions and results, we have decided to present here one specific setup which allows for relatively neat conditions. However, we transparently outline the main steps and bounds needed for the proofs in Section~\ref{ss:proofTH}. From there it is not very complicated to see where the respective conditions enter and how changing the assumptions will change the results. Here we refer to the decomposition of the approximation error in \eqref{e:bt} and \eqref{e:buni} in particular. This leaves some flexibility and allows the analyst to adapt the theorems to the setup under investigation.

\begin{theorem}\label{thm2}
Let Assumptions~\ref{a:noise}--\ref{a:pcs} hold. Assume that $p=p(T)$ diverges at a subexponetial rate and that $L=L(T)$ diverges at a subpolynomial rate.  Furthermore, assume that for some $\nu>0$ and some $\rho>0$ we have $\lambda_j\geq \rho j^{-\nu}$ and that there is some $\beta>0$ such that $p/\hat\gamma_L=O_P(L^{\beta})$. Then, for any $1\leq t\leq T$ it holds that
$$\max_{1 \leq j \leq p } | \hat{X}_t(s_j) - X_t(s_j)|= O_P\left(L^{2\beta+5/2+\nu/2}\bigg(\frac{1}{T^{1/4}}+\frac{T^{1/4}}{\sqrt{p}}\bigg)\right).$$
Moreover, if $E\|X_1\|^{q}<\infty$ for some $q>4$, then for all $q^\pr <q$ we have
$$\max_{1 \leq t \leq T }\max_{1 \leq j \leq p } | \hat{X}_t(s_j) - X_t(s_j)|=o_P\left(L^{2\beta+5/2+\nu}T^{1/q^\pr}\bigg(\frac{1}{T^{1/4}}+\frac{T^{1/4}}{\sqrt{p}}\bigg)\right).$$
\end{theorem}
 
 As one referee pointed out, it may appear counter-intuitive that  the reconstruction error vanishes more slowly when the eigenvalue sequence decreases faster. The problem here is that the smaller the eigenvalue, the weaker the signal associated with the corresponding factor and the more difficult it is to estimate. A similar problem occurs in functional regression. While a fast decay of the eigenvalues associated to the functional covariate indicates that we are ``almost" in a finite dimensional setting, convergence rates in fact become slower. See, for example, the results in \citet{hallhorowitz:2007}.

The technical conditions in Theorem~\ref{thm2} are mild. But as we noted before, they can easily be generalized. For example, it is possible to assume a polynomial rate for $L(T)$, but then we will get weaker and more complicated error terms. 
One part of our assumption which is less transparent is $p/\hat\gamma_L=O_P(L^{\beta})$. The following lemma provides a sufficient condition. 
\begin{lemma}\label{l:evb}
Suppose that we have positive parameters $\rho,\rho^\pr,\nu$ and $\nu^\pr$ such that  $\rho^\pr j^{-\nu^\pr} \geq\lambda_j\geq \rho j^{-\nu}$, for all $j\geq 1$. Moreover, assume that Lemma~\ref{l:lambdagamma} applies. Then if $L^{2\nu}/\min\{p,T\}\to 0$, we have $p/\hat\gamma_L=o_P(L^{\nu^\pr})$.
\end{lemma}
 
In Theorems \ref{thm} and \ref{thm2} the required number of factors $L$ is assumed to be known, even though it may grow with $T$ within the setting of Theorem 2. Naturally, in practice $L$ is not known and must be estimated before applying the above mentioned methods.  
Determining the number of factors is a separate problem, which goes beyond the scope of our article. We refer e.g.,\ to \citet{baing2002} or \citet{onatski2010}. In our practical implementations we found the bi-cross validation method suggested by \citet{owenwang2016} to be most effective. However, we show that the use of an appropriate estimator for $L$ does not interfere with our results. Let us assume here that an estimator $\hat{L}$ exists, such that
\begin{equation}\label{e:consist}
\Pr(|\hat{L} - L|= 0) \longrightarrow 1.
\end{equation}
\begin{corollary}\label{c:1}
Suppose that the number of factors is estimated by $\hat L$, and that this estimator satisfies \eqref{e:consist}. Then the statements of Theorems \ref{thm} and \ref{thm2}  remain valid.
\end{corollary}

We conclude this section by commenting on the dependence structure of the estimated signals $(\hat X_t(\bs))$. As one of the reviewers pointed out, our procedure to some extent modifies the autocorrelation structure of the functional series. For example, if the $X_t$ were independent, then the $\hat X_t$ may nevertheless be dependent, since each $\hat X_t$ is computed from the entire sample. Intuitively, this effect will be asymptotically negligible by our main results. Let us  illustrate that this intuition is correct on the basis of the empirical autocovariance. To this end we note that
\begin{align*}
&\left|\frac{1}{T}\sum_{t=1}^{T-|h|}\left(\hat X_{t+|h|}(s_i)\hat X_t(s_j)- X_{t+|h|}(s_i) X_t(s_j)\right)\right|\\
&\quad \leq\frac{1}{T}\sum_{t=1}^{T-|h|}\left|X_{t}(s_j)\right|\left|\hat X_{t+|h|}(s_i)- X_{t+|h|}(s_i)\right|
+\frac{1}{T}\sum_{t=1}^{T-|h|}\left|\hat X_{t+|h|}(s_i)\right|\left|\hat X_t(s_j)- X_t(s_j)\right|.
\end{align*}
The first term on the right hand side can be bounded by
$$
\frac{1}{T}\sum_{t=1}^{T-|h|}\left|X_{t}(s_j)\right|\left|\hat X_{t+|h|}(s_i)- X_{t+|h|}(s_i)\right|\leq \max_{1\leq i\leq p}\max_{1\leq t\leq T}\left|\hat X_{t}(s_i)- X_{t}(s_i)\right|\times
\frac{1}{T}\sum_{t=1}^{T}\max_{1\leq j\leq p}\left|X_{t}(s_j)\right|.
$$
The second term can be treated similiarly. The factor $T^{-1}\sum_{j=1}^{T}\max_{1\leq j\leq p}\left|X_{t}(s_j)\right|$ is $O_P(1)$ by Assumption 2 (iii) and the ergodic theorem, while the first factor tends to zero at the rate specified by our theorems. On the other hand, it follows from the ergodic theorem that the empirical acf from the latent signals $X_t(s_i)$ is consistent for $\mathrm{Cov}(X_{t+|h|}(s_i),X_t(s_j))$. The imposed weak dependence assumption can be used to obtain rates.

\subsection{Comparison to existing results}\label{ss:comp}
Let us begin with noting that the common curve-by-curve fitting techniques, such as spline smoothing or non-linear regression techniques, don't profit from $T\to\infty$.  The estimation error will only decrease with growing $p$ and not with a growing sample size.  Hence our results do not directly compare to these methods. 

A method taking into account the entire sample is \cite{Staniswalis:Lee:1998}.  Their approach, in contrast to ours, returns a full curve and doesn't require a fixed sampling design. Of course, we can also obtain a full curve by using a sensible interpolation scheme, e.g.\ $\hat X_t(s):=X_t(s_i)$ for all $s\in [s_i,s_{i+1})$ is a simple solution. We think, however, that such an extension of $\hat X_t(\bs)$ to a curve $X_t(s)$ is of little practical relevance, since for the processing of real data we will most often use the discretised curves anyway.  For example, in a functional regression model $Y_t= \int\beta(s)X_t(s)ds+\epsilon_t$, the integral typically can only be computed numerically, e.g.\ by applying a composite trapezoidal rule. It is then natural to use $\bs$ as the nodes.

Extending our results to obtain a theoretical bound for $\sup_s|\hat X_t(s)-X_t(s)|$ is conceptually not very complicated.  Set $h=\max_{1\leq i\leq p}|s_{i+1}-s_i|$ and define the modulus of continuity of $X_t$ as $\delta^X(h)=\sup_{s,s'\in [0,1], |s-s'|\leq h}  |X_t(s)-X_t(s')|$, then
$$
\sup_{s\in [0,1]}|\hat X_t(s)-X_t(s)|\leq \max_{1\leq i\leq p}|\hat X_t(s_i)-X_t(s_i)|+\delta^X(h).
$$
For the first term on the right we have provided bounds. By imposing suitable assumptions on the path properties of the signal $X_t$, the errors $\delta^X(h)$ can be controlled, provided the sampling grid $\bs$ is dense. Let us remark at this point that for our theorems we don't require a dense grid $\bs$, i.e.,\ we can also deal with a situation where $h\not\to 0$ as $p\to\infty$.

Corresponding asymptotic results for the method of \cite{Staniswalis:Lee:1998} have been established in \cite{Mueller:Stadtmueller:Yao:2006}.  These authors also assume a fixed and dense sampling design but indicate that such assumptions can be relaxed to a certain extent. Their framework requires smoothness of the underlying signal in terms of higher order derivatives. This is an assumption which we particularly seek to avoid. In \cite{hormann:jammoul:2020a} we have seen that smoothness of the signal is not just a technical assumption needed for their proof, but also that the practical performance suffers severely for signals which are not very smooth. An important technical restriction in \cite{Mueller:Stadtmueller:Yao:2006} is that signals are assumed to be bounded, which then also excludes the important case of Gaussian processes. In our paper we only request $>4$ order moments. 

The bounds in \cite{Mueller:Stadtmueller:Yao:2006} are quite involved and of the form $\sup_s|\hat X_t(s)-X_t(s)|$. To the best of our knowledge, we are the first paper to provide uniform bounds over $s$ and $1\leq t\leq T$. 

Basically all papers considering the setting \eqref{signoise} assume that the noise components are i.i.d.\ or at least white noise. We think that this is a severe restriction. In real data it is rather plausible that errors of subsequent measurements are likely to be correlated.

From the factor model perspective, we have improved on the conditions used in \citet{fanetal2013}, who require, among others, exponential moments for the involved processes. Moreover, we allow $L\to\infty$, which is also novel.

\section{Main parts of the proofs}\label{ss:proofTH}
We define  $U=(u_1,\ldots, u_T)$ and $F^\top=(f_1,\ldots, f_T)$. Moreover, let $b_j^\top$ be the $j$-th row of 
$$
B:=(\sqrt{\lambda_1}\varphi_1(\bs),\ldots,\sqrt{\lambda_L}\varphi_L(\bs)).
$$ 
The $j$-th column of $Y^\top$ and $U^\top$ are denoted $Y_j$ and $U_j$,  respectively. The $j$-th row of $F^\top$ is denoted by $F_j$.
 Then we have by \eqref{signoise}, \eqref{e:mercer} and Assumption~\ref{a:signal} (iv) that
\begin{equation}\label{e:repmat}
Y=BF^\top +U.
\end{equation}
In this notation, the objective is to estimate $BF^\top$ through some estimator $\hat B\hat F^\top$ and to show that $BF^\top-\hat B\hat F^\top$ tends to zero, uniformly over each component. 
In particular we let $\hat{F} = \sqrt{T} \hat{E}$ and $\hat{B} = T^{-1}Y\hat{F}$, where $\hat{E} = (\hat{e}_1, \ldots, \hat{e}_L)$ denotes the eigenvectors associated to the $L$ largest eigenvalues of $T^{-1}Y^\top Y$. Then $\hat{X} = \hat{B}\hat{F}^\top$ as in \eqref{e:pcaapproach}. We introduce the matrices $\hat\Lambda=\mathrm{diag}(\hat\gamma_1,\ldots, \hat\gamma_L)$  and  $H:=T^{-1}\hat\Lambda^{-1}\hat F^\top F B^\top B.$ We will see in Lemma \ref{l:H} that $H$ is asymptotically orthogonal, which means
$
BF^\top-\hat B\hat F^\top\approx BH^\top H F^\top-\hat B\hat F^\top.
$
The basic idea is then to show that
$\hat B-BH^\top$ and $\hat F^\top -H F^\top$
become small. The matrix $H$ guarantees that the estimated and the empirical scores share the same orientation.

More specifically we have $\hat X_t(s_j)-X_t(s_j)=\hat b_j^\top\hat f_t- b_j^\top f_t$ and we write
\begin{align*}
\hat b_j^\top\hat f_t- b_j^\top f_t= (\hat b_j-H b_j)^\top (\hat f_t -Hf_t) +(\hat b_j- Hb_j)^\top Hf_t+(Hb_j)^\top(\hat f_t-Hf_t) +b_j^\top (H^\top H -I_L)f_t.
\end{align*}
Thus we have 
\begin{align}
\max_{j}|\hat X_t(s_j)-X_t(s_j)|&\leq R^{(1)}R^{(2)}+R^{(2)}R^{(5)}\|f_t\|+R^{(5)}R^{(3)}R^{(1)}+R^{(3)}R^{(6)}\|f_t\|,\label{e:bt}\\
\max_{j,t}|\hat X_t(s_j)-X_t(s_j)|&\leq R^{(1)}R^{(2)}+R^{(2)}R^{(5)}R^{(4)}+R^{(5)}R^{(3)}R^{(1)}+R^{(3)}R^{(6)}R^{(4)},\label{e:buni}
\end{align} 
where
$R^{(1)}:=\max_{1\leq t\leq T}R^{(1)}_t$ with $R^{(1)}_t:=\|\hat f_t -Hf_t\|;$
$R^{(2)}:=\max_{1\leq j\leq p}R^{(2)}_j$ with $R^{(2)}_j:=\|\hat b_j-H b_j\|;$
$R^{(3)}=\max_{1\leq j\leq p}\|b_j\|;$
$R^{(4)}:=\max_{1\leq t\leq T}\|f_t\|;$
$R^{(5)}:=\|H\|;$ and finally
$R^{(6)}:=\|H^\top H -I_L\|.$

Here---with a slight abuse of notation---$\|\cdot\|$ denotes next to the $L^2$-norm the Euclidean norm and the spectral norm of a matrix. The concrete meaning will be clear from the context. Bounds for $R^{(i)}$ will be obtained below in a sequence of lemmas. The proofs of these technical lemmas are given in \ref{a:proofs}.

We begin with the term $R^{(1)}$. 
By simple algebraic manipulations it follows from \eqref{e:repmat} that
\begin{align*}
\hat F^\top -H F^\top&=\hat F^\top-\hat\Lambda^{-1}\hat F^\top\frac{1}{T}FB^\top BF^\top = \hat F^\top-\hat\Lambda^{-1}\hat F^\top\frac{1}{T}(Y^\top Y-U^\top U-FB^\top U-U^\top B F^\top)\\
&=\hat F^\top-\hat\Lambda^{-1}\hat F^\top \frac{1}{T}Y^\top Y+\frac{1}{T}\hat\Lambda^{-1}\hat F^\top(U^\top U+FB^\top U+U^\top B F^\top).
\end{align*}
Using the spectral theorem we get the representation $T^{-1}Y^\top Y=\hat E \hat \Lambda \hat E^\top$ $ + \hat M^c$, with $\hat E^\top \hat M^c=0$. Hence 
$
\hat F^\top-T^{-1}\hat\Lambda^{-1}\hat F^\top Y^\top Y=0.
$
We thus obtain
$
\hat F^\top -H F^\top=A_1+A_2+A_3,
$
where 
\begin{align*}
A_1=\frac{1}{T}(\hat \Lambda/p)^{-1}\hat F^\top \frac{U^\top U}{p},\quad
A_2=\frac{1}{T}(\hat \Lambda/p)^{-1}\hat F^\top \frac{U^\top B F^\top}{p},\quad
A_3=\frac{1}{T}(\hat \Lambda/p)^{-1}\hat F^\top \frac{FB^\top U}{p}.
\end{align*}
Notice that $\hat F^\top -H F^\top$ is an $L\times T$ random matrix. The $t$-th column can be written as
\begin{equation}\label{e:at}
\hat f_t -H f_t=\frac{1}{T}(\hat\Lambda/p)^{-1}\hat F^\top\left(\frac{U^\top u_t}{p}+\frac{U^\top B f_t}{p}+\frac{FB^\top u_t}{p}\right)=:A_{1t}+A_{2t}+A_{3t}.
\end{equation}
The following lemma provides the order of magnitude of $\max_{1\leq t\leq T}\|A_{mt}\|$, $m\in \lbrace 1,2,3 \rbrace$. 
\begin{lemma}\label{l:ak11}
Under Assumptions~\ref{a:noise} and~\ref{a:signal} we have
\begin{align}
\frac{\hat\gamma_L}{p}\max_{t=1}^T \|A_{1t}\|&=O_P\left(\sqrt{L}\left(1/T^{1/4}+T^{1/4}/\sqrt{p}\right)\right),\label{e:ak111}\\
\frac{\hat\gamma_L}{p}\max_{t=1}^T \|A_{2t}\|&=O_P\left(\sqrt{L}T^{1/4}/\sqrt{p}\right),\label{e:ak112}\\
\frac{\hat\gamma_L}{p}\max_{t=1}^T \|A_{3t}\|&=O_P\left(\sqrt{L}T^{1/4}/\sqrt{p}\right),\label{e:ak113}
\end{align}
as $T\to\infty$. Hence, $R^{(1)}=O_P\left( (p/ \hat\gamma_L) \sqrt{L}\left(1/T^{1/4}+T^{1/4}/\sqrt{p}\right)\right)$.
\end{lemma}
Next we study the order of magnitude of
$R^{(2)}$.
By the definition of $\hat B$ and $H$ we get
$$
\hat b_j-H b_j=\frac{1}{T}\hat F^\top Y_j-Hb_j=\frac{1}{T}\hat F^\top Y_j-\frac{1}{T}HF^\top(Y_j-U_j)+\left(\frac{1}{T}HF^\top(Y_j-U_j)-Hb_j\right).
$$
Rearranging the terms we thus get
$$
\hat b_j-H b_j=\frac{1}{T}\left(\hat F^\top-H F^\top\right)Y_j+\frac{1}{T}HF^\top U_j + H\left(\frac{F^\top F}{T}-I_L\right)b_j=:B_{1j}+B_{2j}+B_{3j}.
$$
The following lemma provides the order of magnitude of $\max_{j=1}^p\|B_{mj}\|$, $m \in \lbrace 1,2,3 \rbrace$. 
\begin{lemma}\label{l:bj} 
Consider Assumptions \ref{a:noise}, \ref{a:signal} and \ref{a:pcs}. Suppose that $L/(T\lambda_L^2)\to 0$  and that $p$ diverges at a subexponential rate. Then, as $T\to\infty$,
\begin{align}
\frac{\hat\gamma_L}{p}\max_{j=1}^p \|B_{1j}\|&=O_P\left(L\left(1/\sqrt{T}+ 1/\sqrt{p}\right)\right),\label{e:bj1}\\
\frac{\hat\gamma_L}{p}\max_{j=1}^p \|B_{2j}\|&=O_P\left(\frac{L^{3/2}}{\sqrt{\lambda_L}}\sqrt{\frac{\log(pL)}{T}}\right),\label{e:bj2}\\
\frac{\hat\gamma_L}{p}\max_{j=1}^p \|B_{3j}\|&=O_P\left(\frac{L^2}{\lambda_L^{3/2}\sqrt{T}}\right).\label{e:bj3}
\end{align}
Thus, if $m_T$ denotes the maximum of the rates in \eqref{e:bj1}--\eqref{e:bj3}, then $R^{(2)}=O_P\left(m_T(p/\hat\gamma_L\right)).$
\end{lemma}

\begin{lemma}\label{l:bs} 
Under Assumption~\ref{a:signal} (ii), $R^{(3)}\leq C_X^{1/4}.$
\end{lemma}

\begin{lemma}\label{l:Hbound}
Under Assumptions~\ref{a:signal} and \ref{a:pcs} 
$$
R^{(5)}=O_P\left(\frac{p}{\hat\gamma_L}\frac{L}{\sqrt{\lambda_L}}\right).
$$
\end{lemma}

\begin{lemma}\label{l:H}
Under Assumptions~\ref{a:noise}, \ref{a:signal} and \ref{a:pcs}, as $T\to\infty$, 
\begin{align}
R^{(6)} &= O_P\left(\left(\frac{p}{\hat\gamma_L}\right)^4\frac{L^5}{\lambda_L^3}\left(
\frac{1}{\sqrt{T}}+\frac{1}{\sqrt{p}}
\right)\right).\label{e:HH2}
\end{align}
\end{lemma}

\begin{proof}[\textbf{\upshape Proof of Theorem~\ref{thm}:}]
The proof is immediate from Lemmas~\ref{l:ak11}--\ref{l:H}. The assumption $p/\hat\gamma_L=O_P(1)$ implies that $R^{(1)}$ and $R^{(2)}$ are both $O_P\left(1/T^{1/4}+T^{1/4}/\sqrt{p} \right)$. The terms $R^{(3)}$, $R^{(5)}$ and $\|f_t\|$ are $O_P(1)$  and $R^{(6)}=O_P\left(
1/\sqrt{T}+1/\sqrt{p}
\right)$ . Now use \eqref{e:bt} and \eqref{e:buni}.
\end{proof}

\begin{proof}[\textbf{\upshape Proof of Theorem~\ref{thm2}:}]
We use again Lemmas~\ref{l:ak11}--\ref{l:H}. With our assumption $p/\hat\gamma_L=O_P(L^\beta)$ and the lower bound for the eigenvalues we get that $R^{(1)}=O_P\left(L^{\beta+1/2}(1/T^{1/4}+T^{1/4}/\sqrt{p}) \right)$ and $R^{(2)}=O_P\left(L^{\beta+1}(1/T^{1/4}+T^{1/4}/\sqrt{p}) \right)$. The term $R^{(3)}$ is bounded and by the Markov inequality $\|f_t\|=O_P(\sqrt{L})$. Finally, $R^{(5)}=O_P(L^{\beta+\nu/2+1})$ and $R^{(6)}=O_P\big(L^{4\beta+3\nu+5}(1/\sqrt{T}+1/\sqrt{p})\big)$. Now use \eqref{e:bt} for the non-uniform bound. Since we assume a subpolynomial rate for $L$, we see that the dominating term is $R^{(2)}R^{(5)}\|f_t\|$. For the uniform bound we use \eqref{e:buni}. Here the dominating term is $R^{(2)}R^{(5)}R^{(4)}$. A simple generalisation of Lemma~\ref{l:fs} yields that for any $q^\pr <q$ we have $R^{(4)}=o_P(L^{1/2+\nu/2}T^{1/q^\pr})$.
\end{proof}

\begin{proof}[\textbf{\upshape Proof of Corollary~\ref{c:1}:}]
We let $\hat X_t(s_i|K)$ be the estimated components in \eqref{e:pcaapproach} with $\hat E=(\hat e_1,\ldots, \hat e_K)$.  When $L$ is unknown our estimator is of the form $\hat X_t(s_i|\hat L)$. Moreover, $\hat X_t(s_i|L)=\hat X_t(s_i)$ as it was used in Theorems~\ref{thm} and \ref{thm2}. From those theorems we have statements of the form \newline
$\Pr\left(\max_{1\leq t\leq T}\max_{1\leq j\leq p}|\hat X_t(s_i| L)-X_t(s_i)|> a_T\right)\to 0.$
This remains true if we replace $L$ by $\hat L$, since
\begin{align*}
\Pr\left(\max_{1\leq t\leq T}\max_{1\leq j\leq p}|\hat X_t(s_i|\hat L)-X_t(s_i)|> a_T\right)
\leq \Pr\left(\max_{1\leq t\leq T}\max_{1\leq j\leq p}|\hat X_t(s_i|  L)-X_t(s_i)|> a_T\right) +\Pr\left(\hat L\neq L\right)\to 0.
\end{align*}
\end{proof}

\section{Conclusion}\label{s:conclusion}

In this paper we develop the theoretical foundation to a factor model based approach for preprocessing functional data. We show that the recovery of the underlying signal from noisy functional data observations can be done consistently using a simple and straight-forward application of a principal component factor model. We relax many assumptions on approximate factor models as well as comparable FDA smoothing techniques and provide a general and neat setting which appears realistic in real life applications. We extend common factor analysis approaches by considering the case where the number of factors may grow with the dimensions of the dataset. Furthermore, we expand upon comparable FDA smoothing techniques by including the whole dataset in the signal recovery instead of applying common curve-by-curve techniques. Practical applications as well as extensive simulation studies are presented in the companion paper \cite{hormann:jammoul:2020a}. 
\appendix

\section{Proofs of technical lemmas}
\label{a:proofs}

In this appendix we prove the lemmas stated in Section~\ref{ss:proofTH} and some further technical lemmas. We start by reviewing some further background and notation. 

All random variables below are defined on a common probability space $(\Omega, \mathcal{A},P)$. We note that by our assumption the processes $X_t$ are elements in $H=L^2([0,1])$---the space of square integrable functions on the interval $[0,1]$. This space is equipped with scalar product $\langle x, y\rangle=\int_0^1x(t)y(t)dt$ and norm $\|x\|=\sqrt{\langle x,x\rangle}$.  For the sake of a light notation we will use $\|\cdot\|$ also for the Euclidean norm and the spectral norm of a matrix. It will be clear from the context what type of norm it indicates. We use $\|\cdot\|_F$ for the Frobenius norm.

The signals $X_t$ can also be viewed as elements in $L^2_H=L^2_H(\Omega,\mathcal{A},P)$, i.e.\ the space of $H$-valued random elements $X$ with $E\|X\|^2<\infty$. The space $L^2_H$ is also a Hilbert space with inner product $E\langle X,Y\rangle$.

The outer product $\otimes$ between two functions $x,y\in H$ defines a Hilbert-Schmidt (HS) operator  through $x\otimes y(z)=x\langle z, y\rangle$. The set of Hilbert-Schmidt (HS) operators on $H$ will be denoted by $\mathcal{S}$. Let $(e_i)$ be an orthonormal basis (ONB) of $H$. Then the space $\mathcal{S}$ equipped with the scalar product $\langle W, V\rangle_\mathcal{S}:=\sum_{i\geq 1}\langle W(e_i), V(e_i)\rangle$ (which can be shown to be independent of the ONB) is another separable Hilbert space. It can be easily seen that $E\|X\|^4<\infty$, implies $E\|X\otimes X\|_\mathcal{S}^2<\infty$ and thus $X\otimes X\in L^2_\mathcal{S}(\Omega,\mathcal{A},P)$. Then we may define $\Omega^X:=\mathrm{Var}(X\otimes X)$. For the definition of covariance operators on Hilbert spaces we refer to \cite{bosq:2000}.

 We use $\stackrel{d}{\to}$ to indicate convergence in distribution. A normal random element in some separable Hilbert space $H$ with mean $\mu$ and covariance operator $\Omega$ is denoted by $\mathcal{N}_H(\mu,\Omega)$.

\subsection*{Preliminary lemmas}

\begin{lemma}\label{l:XtimesX} Define $W_T:=\sum_{t=1}^T X_t\otimes X_t$.  Under Assumption~\ref{a:signal}~(i) 
\begin{enumerate}[label=(\roman*), before=\itshape,font=\normalfont]
\item $\frac{1}{\sqrt{T}}\big(W_T-T\,\Gamma^X\big)\stackrel{d}{\to} \mathcal{N}_\mathcal{S}(0,\Omega^X),$
\item $E\big\|W_T-T\,\Gamma^X\big\|_\mathcal{S}^2\leq c_1 \, T$, where $c_1$ is independent of $T$.
\end{enumerate}
\end{lemma}
\begin{proof}[\textbf{\upshape Proof:}]
Assumption~\ref{a:signal}~(i) implies that the process $(X_t\otimes X_t\colon t\geq 1)$ is $L^2$-$m$-approximable in $\mathcal{S}$. This follows by a simple modification of the proof of Lemma~2.1 in \cite{hoermann2010}. Thus Theorem~8 in \cite{hormann:cerovecki:2017} applies and this in turn yields the central limit theorem (see Theorem~5 in \cite{hormann:cerovecki:2017}). Part (ii) follows directly from Theorem~3.1 in \cite{hoermann2010}.
\end{proof}

\begin{lemma}\label{l:noisemoments} 
Under Assumption~\ref{a:noise} we have a constant $c_2$ which is independent of $p$ such that for $k\in \lbrace 1,2 \rbrace$
\begin{enumerate}[label=(\roman*), before=\itshape,font=\normalfont]
\item $E\big(u_2^\top u_1\big)^{2k}\leq c_2\, p^k$,
\item $E\|u_1\|^{4k}\leq c_2\, p^{2k}$.
\end{enumerate}
\end{lemma}
\begin{proof}[\textbf{\upshape Proof:}]
We show (i). The proof of (ii) is along the same lines. The main ingredient is Isserlis' theorem, which states that for a zero-mean multivariate normal vector $(Z_1,\ldots, Z_{2k})^\top$ the product moments are given by
\newcommand{\mfp}{\mathfrak{p}}
\newcommand{\mcP}{\mathcal{P}}
\begin{equation}\label{e:isserlis}
EZ_1 \cdots Z_{2k}=\sum_{\mfp\in \mcP[1,\ldots, 2k]} \prod_{\{i,j\}\in \mfp} \mathrm{Cov}(Z_i, Z_j),
\end{equation}
where $\mcP[I]$ denotes the set of all pairings of the components of some index vector $I$. Denote $u_1=(\varepsilon_1,\ldots, \varepsilon_p)^\top$.  Since $u_1$ and $u_2$ are i.i.d.\ we get
$
E\big(u_2^\top u_1\big)^{2k}=\sum_{i_1,\ldots, i_{2k}=1}^p(E\varepsilon_{i_1}\ldots \varepsilon_{i_{2k}})^2,
$
and thus
\begin{align*}
E\big(u_2^\top u_1\big)^{2k}&=\sum_{i_1,\ldots, i_{2k}=1}^p\left(\sum_{\mfp\in \mcP[(i_1,\ldots, i_{2k})]} \prod_{\{\ell,j\}\in \mfp} \gamma^u(\ell-j)\right)^2 \\
&\leq \sum_{i_1,\ldots, i_{2k}=1}^p\# \mcP[(i_1,\ldots, i_{2k})]\sum_{\mfp\in \mcP[(i_1,\ldots, i_{2k})]} \prod_{\{\ell,j\}\in \mfp} \big(\gamma^u(\ell-j)\big)^2\\
&=(\#\mcP[(1,\ldots, 2k)])^2\left(\sum_{\ell,j=1}^p\big(\gamma^u(\ell-j)\big)^2\right)^k \leq \big((2k-1)!!\big)^2\big(C_u \gamma^u(0)\big)^k p^k.
\end{align*}
The last inequality was deduced from Assumption~\ref{a:noise}.
\end{proof}

\begin{lemma}\label{l:FFI}
Under Assumption~\ref{a:signal}~(i)
$$
\left\|\frac{1}{T}F^\top F-I_L\right\|_F=O_P\left( \frac{L}{\lambda_L \sqrt{T}}\right).
$$
\end{lemma}

\begin{proof}[\textbf{\upshape Proof:}]
Denoting by $\delta_{ij}$ the Kronecker delta, we notice that the $i,j$-th entry of the $(L\times L)$-matrix $T^{-1}F^\top F-I_L$ can be written as
$$
\frac{1}{T}\sum_{t=1}^T f_{it}f_{jt}-\delta_{ij}=\frac{1}{\sqrt{T}}\frac{1}{\sqrt{\lambda_i\lambda_j}}
\bigg\langle\frac{1}{\sqrt{T}}\sum_{t=1}^T (X_t\otimes X_t-\Gamma^X)(\varphi_i),\varphi_j\bigg\rangle.
$$
 This shows that
$$
\left\|\frac{1}{T}F^\top F-I_L\right\|_F\leq \frac{1}{\sqrt{T}}\left\|\frac{1}{\sqrt{T}}\big(W_T-T\,\Gamma^X\big)\right\|_\mathcal{S} \left(\sum_{i,j=1}^L\frac{1}{\lambda_i\lambda_j}\right)^{1/2}.
$$
By Lemma~\ref{l:XtimesX}~(i) we infer that $\|T^{-1/2}\big(W_T-T\,\Gamma^X\big)\|_\mathcal{S}=O_P(1)$ and hence the result follows.
\end{proof}

For the next result we note that the $k$-th row of $\hat f_{t} -H f_{t}$ is given by
\begin{equation}\label{e:atk}
\frac{1}{T}\frac{p}{\hat\gamma_k}\hat F_k\left(\frac{U^\top u_t}{p}+\frac{U^\top B f_t}{p}+\frac{FB^\top u_t}{p}\right)=:A_{1tk}+A_{2tk}+A_{3tk},
\end{equation}
where $\hat F_k$ is the $k$-th row of $\hat F^\top$.

\begin{lemma}\label{l:ak12}
Under Assumptions~\ref{a:noise} and \ref{a:signal} 
\begin{align}
\sum_{k=1}^L\frac{1}{T}\sum_{t=1}^T A_{1tk}^2&=O_P\left(L^{2}(\hat{\gamma}_L/p)^{-2}\big(1/T+1/p\big)\right),\label{e:ak121}\\
\sum_{k=1}^L\frac{1}{T}\sum_{t=1}^T A_{2tk}^2&=O_P\left(L^{2}(\hat{\gamma}_L/p)^{-2}/p\right),\label{e:ak122}\\
\sum_{k=1}^L\frac{1}{T}\sum_{t=1}^T A_{3tk}^2&=O_P\left(L^{2}(\hat{\gamma}_L/p)^{-2}/p\right).\label{e:ak123}
\end{align}
\end{lemma}

\begin{proof}[\textbf{\upshape Proof:}]
We show \eqref{e:ak121} and \eqref{e:ak122}. For \eqref{e:ak121} we use that
\begin{align*}
(\hat\gamma_k/p)^2\frac{1}{T}\sum_{t=1}^T A_{1tk}^2 = \frac{1}{T}\sum_{t=1}^T\frac{1}{T^2p^2}\left(\hat F_k U^\top u_t\right)^2 =\frac{1}{T}\sum_{t=1}^T\frac{1}{T^2p^2}\left(\sum_{\ell =1}^T\hat f_{k\ell} u_\ell^\top u_t\right)^2 \leq \frac{1}{T^2}\sum_{t=1}^T\sum_{\ell=1}^T\left(\frac{u_\ell^\top u_t}{p}\right)^2.
\end{align*}
In the last step we used the Cauchy-Schwarz inequality and the fact that $\hat f_{k\ell}$ ($\ell$-th column and $k$-th row of $\hat F^\top$) satisfies $\sum_{\ell=1}^T\hat f_{k\ell}^2=T$. 
Using the Markov inequality with Assumption~\ref{a:noise} and Lemma~\ref{l:noisemoments} we get for any $\kappa>0$ 
\begin{align*}
\Pr\left(\frac{1}{T^2}\sum_{t=1}^T\sum_{\ell=1}^T\left(\frac{u_\ell^\top u_t}{p}\right)^2>\kappa \big(1/T+1/p\big)\right) &\leq \Pr\left(\frac{1}{T}\sum_{t=1}^T\left(\frac{u_t^\top u_t}{p}\right)^2>\kappa\right)+\Pr\left(\frac{1}{T^2}\sum_{t=1}^T\sum_{\ell=1\atop \ell\neq t}^T\left(\frac{u_\ell^\top u_t}{p}\right)^2>\kappa /p\right)\\
&\leq E\left(\frac{u_1^\top u_1}{p}\right)^2\Big\slash\kappa+\frac{T-1}{T}E\left(\frac{u_1^\top u_2}{p}\right)^2\frac{p}{\kappa}\leq \frac{c_2\left(1+\frac{T-1}{T}\right)}{\kappa}.
\end{align*}

Thus 
\begin{align*}
\Pr\left(\sum_{k=1}^L\frac{1}{T}\sum_{t=1}^T A_{1tk}^2>L^2(\hat\gamma_L/p)^{-2}\big(1/T+1/p\big)\kappa\right) &\leq \Pr\left(\sum_{k=1}^L(\hat\gamma_k/p)^{2}\frac{1}{T}\sum_{t=1}^T A_{1tk}^2>L^2\big(1/T+1/p\big)\kappa\right)\\
&\leq \sum_{k=1}^L \Pr\left((\hat\gamma_k/p)^{2}\frac{1}{T}\sum_{t=1}^T A_{1tk}^2>L\big(1/T+1/p\big)\kappa\right)=O(1/\kappa).
\end{align*}

\noindent The first statement of the lemma is shown. By similar arguments as above we obtain, using $Bf_t = X_t(\bs)$, 
\begin{align*}
(\hat\gamma_k/p)^2\frac{1}{T}\sum_{t=1}^T A_{2tk}^2&\leq \frac{1}{T^2}\sum_{t=1}^T\sum_{\ell=1}^T\left(\frac{u_\ell^\top X_t(\bs)}{p}\right)^2.
\end{align*}
Then, using the independence between $X_t(\bs)$ and $u_\ell$ (Assumption~\ref{a:noise}), we get that
\begin{align*}
E\left(\frac{u_\ell^\top X_t(\bs)}{p}\right)^2&=\frac{1}{p^2}E\left(u_\ell^\top \Gamma^{X}(\bs,\bs) u_\ell\right)\nonumber = \frac{1}{p^2}\sum_{i=1}^p\sum_{j=1}^p \gamma^u(i-j)\Gamma^X(s_i,s_j)\nonumber \\
&\leq \frac{1}{p^2}\sum_{i=1}^p\sum_{j=1}^p |\gamma^u(i-j)|\max_{i,j}|\Gamma^X(s_i,s_j)|\leq C_u\frac{ \sup_{s\in [0,1]}EX^2(s)}{p}.
\end{align*}
By Assumption~\ref{a:signal} (ii) we have $\sup_{s\in [0,1]}EX^2(s)<\infty$ and so \eqref{e:ak122} follows with similar arguments as before using the Markov inequality. The bound for \eqref{e:ak123} can be derived in complete analogy to \eqref{e:ak122}.
\end{proof}

\begin{lemma}\label{l:Bj2}
Under Assumptions~\ref{a:noise} and \ref{a:signal} and if  $L/(T\lambda_L^2)\to 0$ we have 
$$
\max_{j\leq p}\left\|\frac{1}{T}F^\top U_j \right\|=O_P\left(\sqrt{L}\big(\log(pL)/T\big)^{1/2}\right).
$$
\end{lemma}

\begin{proof}[\textbf{\upshape Proof:}]
We have
$$
\max_{j\leq p}\left\|\frac{1}{T}F^\top U_j \right\|=\max_{j\leq p}\left\|\frac{1}{T}\sum_{t=1}^Tf_tu_{tj}\right\|\leq \sqrt{L}\max_{j\leq p \atop i\leq L} \left|\frac{1}{T}\sum_{t=1}^T\frac{1}{\sqrt{\lambda_i}}\langle X_t,\varphi_i\rangle u_{tj}\right|.
$$
Then, we note that for any $\kappa>0$
\begin{align*}
& \Pr\left(\max_{j\leq p \atop i\leq L} \left|\frac{1}{T}\sum_{t=1}^T\frac{1}{\sqrt{\lambda_i}}\langle X_t,\varphi_i\rangle u_{tj}\right|>\kappa\right)\\
&\quad \leq \Pr\left(\max_{j\leq p \atop i\leq L} \left|\frac{1}{T}\sum_{t=1}^T\frac{1}{\sqrt{\lambda_i}}\langle X_t,\varphi_i\rangle u_{tj}>\kappa\right|\cap \max_{k\leq L} \frac{1}{T}\sum_{t=1}^T\frac{1}{ \lambda_k}\langle X_t,\varphi_k\rangle^2\leq 2\right)
+ \Pr\left(\max_{k\leq L}\frac{1}{T}\sum_{t=1}^T\frac{1}{ \lambda_k}\langle X_t,\varphi_k\rangle^2> 2\right) \\
& \quad=: p_{1}(\kappa)+p_{2}.
\end{align*}
For $p_2$ we get with Lemma~\ref{l:XtimesX}~(ii) and the Markov inequality
\begin{align*}
p_2&\leq \sum_{k=1}^L \Pr\left(\bigg\langle\Big(\sum_{t=1}^T  \Big(X_t\otimes X_t-EX_t\otimes X_t \Big)\Big)(\varphi_k),\varphi_k\bigg\rangle> \lambda_k T\right) \\
&\leq \sum_{k=1}^L \Pr\left(\bigg\|\sum_{t=1}^T  \Big(X_t\otimes X_t-EX_t\otimes X_t \Big)\bigg\|_\mathcal{S}> \lambda_k T\right)\\
&\leq \sum_{k=1}^L E\bigg\|\sum_{t=1}^T  W_T-T\Gamma^X \bigg\|_\mathcal{S}^2\bigg/\Big(\lambda_k^2 T^2 \Big)\leq \frac{c_1L}{\lambda_L^2 T}.
\end{align*}
To derive the probability $p_1(\kappa)$ we first condition on $X_1,\ldots,X_T$. Then
\begin{align*}
&\Pr\left(
\max_{j\leq p \atop i\leq L} 
\Big|
\frac{1}{T}\sum_{t=1}^T\frac{1}{\sqrt{\lambda_i}}\langle X_t,\varphi_i\rangle u_{tj}\Big|>\kappa
\cap \max_{k\leq L} \frac{1}{T}\sum_{t=1}^T\frac{1}{ \lambda_k}\langle X_t,\varphi_k\rangle^2\leq 2
\bigg| X_1,\ldots, X_T
\right)\\
&\quad
\leq \sum_{j\leq p \atop i\leq L} \Pr\left(\Big|
\frac{1}{T}\sum_{t=1}^T\frac{1}{\sqrt{\lambda_i}}\langle X_t,\varphi_i\rangle u_{tj}\Big|>\kappa
\cap \max_{k\leq L} \frac{1}{T}\sum_{t=1}^T\frac{1}{ \lambda_k}\langle X_t,\varphi_k\rangle^2\leq 2
\bigg| X_1,\ldots, X_T
\right).
\end{align*}
Conditional on $X_1,\ldots, X_T$ the variable
$$
\frac{1}{T}\sum_{t=1}^T\frac{1}{\sqrt{\lambda_i}}\langle X_t,\varphi_i\rangle u_{tj}\sim N\left(0,\frac{\big(\gamma^u(0)\big)^2}{T}\times \gamma_i^2(X_1,\ldots,X_T)\right),
$$
where $\gamma_i^2(X_1,\ldots,X_T)=T^{-1}\sum_{t=1}^T \lambda_i^{-1}\langle X_t,\varphi_i\rangle^2$. By an elementary bound for the normal distribution function (see e.g., \citet[p.\ 177]{petrov}) we get
$$
\Pr(|N(0,\sigma^2)|>x)<\frac{\sigma}{x}e^{-\frac{x^2}{2\sigma^2}}, \quad x>0.
$$
Thus for any $i$ and $j$ 
$$
\Pr\left(\Big|
\frac{1}{T}\sum_{t=1}^T\frac{1}{\sqrt{\lambda_i}}\langle X_t,\varphi_i\rangle u_{tj}\Big|>\kappa
\cap \max_{k\leq L} \frac{1}{T}\sum_{t=1}^T\frac{1}{ \lambda_k}\langle X_t,\varphi_k\rangle^2\leq 2
\Bigg| X_1,\ldots, X_T
\right)\leq \frac{\sqrt{2}\gamma^u(0)}{\kappa\sqrt{T}}e^{-\frac{T\kappa^2}{4(\gamma^u(0))^2}}.
$$
By taking the expectation of the conditional probability we obtain a bound for the unconditional probability. Then,  summing over $i$ and $j$ yields
$$
p_1(\kappa)\leq \frac{\sqrt{2}\gamma^u(0)\times Lp}{\kappa\sqrt{T}}e^{-\frac{T\kappa^2}{4(\gamma^u(0))^2}}.
$$
Now the proposition of the lemma follows immediately.
\end{proof}

\begin{proof}[\textbf{\upshape Proof of Lemma~\ref{l:lambdagamma}:}]
Without loss of generality we will assume that the sampling points are $\{s_j=j/p\colon 0\leq j\leq p-1\}$.
By the triangular inequality we have
\begin{equation}\label{e:evbound}
|\lambda_i-\hat\gamma_i/p|\leq |\lambda_i-\hat\lambda_i| +|\hat\lambda_i-\hat\lambda_i^*|+|\hat\lambda_i^*-\hat\gamma^X_i/p|+|\hat\gamma^X_i/p-\hat\gamma_i/p|,
\end{equation}
where the terms on the right hand side will now be defined and bounded. 
All the variables appearing in \eqref{e:evbound} are eigenvalues of some covariance matrix or some covariance operator. We consider the covariance matrices 
$$\hat\Sigma^y=\frac{1}{T}YY^\top=\frac{1}{T}\sum_{t=1}^T y_ty_t^\top\quad\text{and}\quad\hat\Sigma^X:=\frac{1}{T}XX^\top=\frac{1}{T}\sum_{t=1}^T X_t(\bs)X_t(\bs)^\top,$$ with eigenvalues $(\hat\gamma_i)$ and $(\hat\gamma^X_i)$, respectively. Thereby we note that the non-zero eigenvalues of $T^{-1}Y^\top Y$  and $\hat\Sigma^y$ coincide. Moreover, we consider the covariance operators $\Gamma^X$, $\hat\Gamma^X$ and $\hat\Gamma^{X^*}$, with eigenvalues $(\lambda_i)$, $(\hat\lambda_i)$ and $(\hat\lambda_i^*)$, respectively. Here $\hat\Gamma^X$ is the empirical covariance operator of the fully observed functional data $X_1(v),\ldots, X_T(v)$ and  $\hat\Gamma^{X^*}$ is the empirical covariance operator of the discretised data $X_t^*(v):=X_t(s_j)$ for $v\in [s_j,s_{j+1})$. Note that $\hat\Gamma^{X^*}(s_k,s_j)=\hat\Gamma^{X}(s_k,s_j)=\hat\Sigma^X_{kj}$.

We start with the first term on the right in \eqref{e:evbound}. It follows from Weyl's inequality that $|\lambda_i-\hat\lambda_i| \leq \|\Gamma^X-\hat\Gamma^X\|$. Hence, by $L^4$-$m$-approximability and Theorem~3.1 in 
\cite{hoermann2010} we get
\begin{equation}\label{e:evb1}
 |\lambda_i-\hat\lambda_i| \leq \|\Gamma^X-\hat\Gamma^X\|=O_P\left(\frac{1}{\sqrt{T}}\right).
\end{equation}
Again with Weyl's inequality and then the Cauchy-Schwarz inequality we obtain
\begin{align*}
 |\hat\lambda_i-\hat\lambda_i^*|^2 &\leq \int\int \left(\frac{1}{T}\sum_{t=1}^T\big(X_t(r)X_t(s)-X^*_t(r)X^*_t(s)\big)\right)^2dsdr \leq \frac{2}{T} \sum_{t=1}^T\left(\big\| X_t\big\|^2+\big\| X_t^*\big\|^2\right)\frac{1}{T}\sum_{t=1}^T\big\|X_t-X^*_t\big\|^2 \\
  &=O_P(1) \sum_{j=0}^{p-1} \int_0^{1/p} \big(X_t(j/p+h)-X_t(j/p)\big)^2dh.
\end{align*}
With \eqref{e:inc} it is then immediate that $|\hat\lambda_i-\hat\lambda_i^*|=O_P\left(1/\sqrt{p}\right)$.

In the next step one needs to study the relationship between the eigenvalues of an operator and a matrix. Henceforth we write $[a]_k$ for the $k$-th component of some vector $a$.  Let $\hat e_i^X$ be the eigenvector associated to $\hat\gamma_i^X$ and let $\hat\phi_i(s)=[\hat e_i^X]_k$ for $s\in[(k-1)/p,k/p)$, $k\in \{1,\ldots,p\}$. Then
\begin{align*}
\hat\gamma_i^X [\hat e_i^X]_k=[\hat\Sigma^X\hat e_i^X]_k&=\sum_{j=1}^p \hat\Gamma^{X}(s_k,s_j)[e_i^X]_j=\sum_{j=1}^p p\int_{s_j}^{s_{j+1}}\hat\Gamma^{X^*}(s_k,s) \hat\phi_i(s)ds.
\end{align*}
Thus, for any $v\in [0,1]$
$$
\int_{0}^{1}\hat\Gamma^{X^*}(v,s) \hat\phi_i(s)ds=\frac{\hat\gamma_i^X}{p}\hat\phi_i(v),
$$
showing that $\hat\gamma_i^X/p=\hat\lambda_i^{*}$.

Finally, 
\begin{align*}
|\hat\gamma_i-\hat\gamma^X_i|&\leq\|\hat\Sigma^y-\hat\Sigma^X\|=\frac{1}{T}\left\|UU^\top+XU^\top+
UX^\top\right\|\leq
 \frac{1}{T}\left\|UU^\top\right\|+\frac{2}{T}\left\|XU^\top\right\|.
\end{align*}
Denote $\Gamma^u=\mathrm{Var}(u_t)$ and $Z=[Z_1,\ldots, Z_T]$ with $Z_t\stackrel{\text{iid}}{\sim} N_p(0,I_p)$. Then
$$ \frac{1}{T}\left\|UU^\top\right\|\stackrel{d}{=}  \frac{1}{T}\left\|(\Gamma^u)^{1/2} ZZ^\top(\Gamma^u)^{1/2}\right\|\leq\|\Gamma^u\|\frac{1}{T}\left\| ZZ^\top\right\|.$$
By a basic inequality for matrix norms, we have $\|\Gamma^u\|\leq \sqrt{\|\Gamma^u\|_{\text{c}}\|\Gamma^u\|_\text{r}}$, where $\|A\|_{\text{c}}$ and $\|A\|_{\text{r}}$ denote the maximal absolute column sum and the maximal absolute row sum of some matrix $A$, respectively. In particular, this implies that $\|\Gamma^u\|\leq \sum_{h\in\mathbb{Z}}|\gamma^u(h)|\leq C_u$ (Assumption~\ref{a:noise}). By results on the distribution of the largest eigenvalue of random covariance matrices in \cite{johnstone:2001} and \cite{karoui:2005}, one can deduce that
$$\frac{1}{T}\left\| ZZ^\top\right\|= \begin{cases}
O_P(1),& \text{  } p/T \rightarrow \gamma\in [0,\infty),\\
O_P(p/T),& \text{  } p/T \rightarrow \infty.
\end{cases}$$
 Next we observe that 
\begin{align*}
\frac{1}{T}\left\|XU^\top\right\| \leq\frac{1}{T}\|X\|\|U^\top\| \stackrel{d}{=} \frac{1}{T}\|X\| \|Z^\top(\Gamma^u)^{1/2}\| \leq\left(\frac{1}{T}\|XX^\top\| \|\Gamma^u\| \frac{1}{T}\|ZZ^\top\|\right)^{1/2}.
\end{align*}
By our previous arguments we see that this last term is $O_P\left(T^{-1}\|XX^\top\| \max\{p/T,1\}\right)^{1/2}$. Moreover,
$$
\frac{1}{T}\|XX^\top\|\leq \mathrm{tr}\left(\hat\Sigma^X\right)=\frac{1}{T}\sum_{t=1}^T\sum_{i=1}^p X_t^2(s_i),
$$
and hence $T^{-1}\|XX^\top\|=O_P(p)$.
Thus, we can conclude
$$
|\hat\gamma_i/p-\hat\gamma^X_i/p|= \begin{cases}
O_P\left(\frac{1}{\sqrt{p}}\right), &\text{ } p/T \rightarrow \gamma\in [0,\infty);\\
O_P\left(\frac{1}{\sqrt{T}}\right), &\text{ } p/T \rightarrow \infty.
\end{cases}
$$
The proof follows.
\end{proof}

\begin{proof}[\textbf{\upshape Proof of Lemma~\ref{l:fs}:}]
It holds by the Bonferroni and the Markov inequality that for any $\epsilon >0$
$$
\Pr\left(\max_{1\leq t\leq T}\|f_t\|>\epsilon T^{1/q^\prime}\right)\leq \sum_{t=1}^T \max_{1\leq t\leq T}\frac{E\|f_t\|^q}{T^{q/q^\prime}\epsilon^q}.
$$
By the stationarity of $(X_t)$ and $\langle X_1, \varphi_i\rangle^2\leq \|X_1\|^2$ we have 
\begin{align*}
E\|f_t\|^{q}=E\left(\sum_{i=1}^L \frac{\langle X_1, \varphi_i\rangle^2}{\lambda_i}\right)^{q/2}\leq E\|X_1\|^{q}\left(L/\lambda_L\right)^{q/2}.
\end{align*}
The statement (i) follows. For statement (ii) we use
$$\Pr(\max_{1\leq t\leq T} \|f_t\| >\epsilon \log T)\leq \sum_{t=1}^T \Pr( \exp(\alpha \|f_t\|) > T^{\alpha\epsilon}).$$
Now using again the Markov inequality and choosing $\epsilon>1/\alpha$, this last term tends to zero.
\end{proof}

\begin{proof}[\textbf{\upshape Proof of Lemma~\ref{l:evb}:}] 
We are going to show that 
\begin{equation}\label{e:stp}
\Pr \left(\left|\frac{1}{\lambda_L}-\frac{p}{\hat\gamma_L}\right|>\frac{1}{\lambda_L}\right)\leq 2 \Pr\left(\left|\lambda_L-\frac{\hat\gamma_L}{p}\right|>\frac{\lambda_L}{2}\right).
\end{equation}
Then, by the lower bound on the eigenvalues we conclude that the right-hand term in \eqref{e:stp} is bounded by \\
 $2 \Pr\left(\left|\lambda_L- \hat\gamma_L/p\right|>\rho L^{-\nu}/2\right).$ Since we assume that Lemma~\ref{l:lambdagamma} applies, and since we assume $L^{2\nu}/\max\{p,T\}\to 0$, this term goes to zero.

 As for \eqref{e:stp} we remark that for a random variable $Y>0$ and some $x>0$, it holds that
\begin{align*}
\Pr\left(\left|\frac{1}{x}-\frac{1}{Y}\right|>\frac{1}{x}\right)
= \Pr\left(|Y-x|>Y\right) \leq \Pr\left(|Y-x|>x/2\right)+\Pr(Y\leq x/2),
\end{align*}
while $\Pr(Y\leq x/2)=\Pr(x-Y\geq x/2)\leq \Pr(|Y-x|\geq x/2)$.

\end{proof}

\begin{proof}[\textbf{ \upshape Proof of Lemma~\ref{l:ak11}:}]
Recall that 
$$
A_{1t}=\frac{1}{T}(\hat\Lambda/p)^{-1}\hat F^\top\frac{U^\top u_t}{p}
=\frac{1}{T}(\hat\Lambda/p)^{-1}\sum_{\ell=1}^T\frac{1}{p}\hat f_\ell u_\ell^\top u_t
$$
and  that $\sum_{\ell=1}^T\|\hat f_\ell\|^2=TL$.
 Thus, by the triangular inequality and the Cauchy-Schwarz inequality it follows that
$$
\|A_{1t}\|\leq \frac{p}{\hat\gamma_L}\frac{1}{T}\left\|\sum_{\ell=1}^T\frac{1}{p}\hat f_\ell u_\ell^\top u_t\right\|\leq 
\frac{p}{\hat\gamma_L} \frac{\sqrt{TL}}{T}\sqrt{\sum_{\ell=1}^T\frac{(u_\ell^\top u_t)^2}{p^2}}.
$$
Hence
\begin{align*}
\frac{\hat\gamma_L}{p}\max_{1\leq t\leq T}\|A_{1t}\|\leq\sqrt{L/T}\left(
\max_{1\leq t\leq T}\|u_t\|^2/p+
\max_{1\leq t\leq T}\sqrt{\sum_{\ell=1\atop \ell\neq t}^T\frac{(u_\ell^\top u_t)^2}{p^2}}\right) =:\sqrt{L/T}\left(\max_{1\leq t\leq T}\|A_{1t}^*\|+\max_{1\leq t\leq T}\|A_{1t}^{**}\|.
\right)
\end{align*}
The Markov inequality yields
\begin{align*}
&P\left(\sqrt{L/T}\max_{1\leq t\leq T}\|A_{1t}^{*}\|>\kappa\right)\leq \kappa^{-4}\frac{L^2}{T^2p^4}E\left[\max_{1\leq t\leq T} \|u_t\|^8\right]\leq  \kappa^{-4}\frac{L^2}{Tp^4}E\left[\|u_1\|^8\right].
\end{align*}
From Lemma~\ref{l:noisemoments}~(ii) we infer that 
$$\sqrt{L/T}\max_{1\leq t\leq T}\|A_{1t}^{*}\|=O_P\Big(\sqrt{L} /T^{1/4}\Big).$$
Yet another application of the Markov inequality and Lemma~\ref{l:noisemoments} gives
\begin{align*}
&\Pr\left(\sqrt{L/T}\max_{1\leq t\leq T}\|A_{1t}^{**}\|>\kappa\right)\leq \kappa^{-4}\frac{L^2}{T^2p^4}E\left[\max_{1\leq t\leq T} \left(\sum_{\ell=1\atop \ell\neq t}^T (u_\ell^\top u_t)^2\right)^2\right]\\
&\qquad 
\leq \kappa^{-4}\frac{L^2}{Tp^4}E\left[\left(\sum_{\ell=2}^T (u_\ell^\top u_1)^2\right)^2\right]
\leq \kappa^{-4}\frac{(T-1)L^2}{Tp^4}E\left[\sum_{\ell=2}^T (u_\ell^\top u_1)^4\right]
\leq \kappa^{-4}\frac{TL^2}{p^4}E\left[(u_2^\top u_1)^4\right]\leq c_2 \kappa^{-4}\frac{TL^2}{p^2}.
\end{align*}
Hence
$$\sqrt{L/T}\max_{1\leq t\leq T}\|A_{1t}^{**}\|=O_P\Big(\sqrt{L}\frac{T^{1/4}}{\sqrt{p}} \Big).$$
Combining these estimates yields \eqref{e:ak111}.

Next we verify \eqref{e:ak112}. We recall that $Bf_t=X_t(\bs)$ and hence
$$
A_{2t}=\frac{1}{T}(\hat\Lambda/p)^{-1}\hat F^\top\frac{U^\top B f_t}{p}=
\frac{1}{T}(\hat\Lambda/p)^{-1}\frac{\sum_{\ell=1}^T\hat f_\ell u_\ell^\top X_t(\bs)}{p}.
$$
Thus
$$
\frac{\hat\gamma_L}{p}\max_{1\leq t\leq T}\|A_{2t}\|\leq\frac{1}{T}
\max_{1\leq t\leq T}\left\|\sum_{\ell=1}^T\hat f_\ell\frac{ u_\ell^\top X_t(\bs)}{p}\right\|\leq
\sqrt{\frac{L}{T}
\frac{1}{p^2}
\max_{1\leq t\leq T}\sum_{\ell=1}^T (u_\ell^\top X_t(\bs))^2
}.
$$
Consequently, by the Markov inequality 
\begin{align*}
&\Pr\left(\frac{\hat\gamma_L}{p}\max_{1\leq t\leq T}\|A_{2t}\|>\kappa\right)\leq \kappa^{-4}\frac{L^2}{T^2p^4}E\left[\max_{1\leq t\leq T}\left(\sum_{\ell=1}^T (u_\ell^\top X_t(\bs))^2\right)^2\right].
\end{align*}
Moreover,
\begin{align*}
&E\max_{1\leq t\leq T}\left(\sum_{\ell=1}^T (u_\ell^\top X_t(\bs))^2\right)^2\leq T \max_{1\leq t\leq T} E
\left(\sum_{\ell=1}^T (u_\ell^\top X_t(\bs))^2\right)^2 = T \max_{1\leq t\leq T} E
\left(\sum_{\ell=1}^T\sum_{j=1}^T (u_\ell^\top X_t(\bs) X_t^\top(\bs)u_j)^2\right)\\
&\qquad = T(T^2-T) \sum_{i_1=1}^p\sum_{i_2=1}^p\sum_{j_1=1}^p\sum_{j_2=1}^p\gamma^u(i_2-i_1)\gamma^u(j_2-j_1)EX_1(s_{i_1})X_1(s_{i_2})X_1(s_{j_1})X_1(s_{j_2})\\
&\qquad \quad+T^2\sum_{i_1=1}^p\sum_{i_2=1}^p\sum_{j_1=1}^p\sum_{j_2=1}^p Eu_{1,i_1}u_{1,i_2}u_{1,j_1}u_{1,j_2}EX_1(s_{i_1})X_1(s_{i_2})X_1(s_{j_1})X_1(s_{j_2}).
\end{align*}
By \eqref{e:isserlis}
$$
Eu_{1,i_1}u_{1,i_2}u_{1,j_1}u_{1,j_2}=\gamma^u(i_1-i_2)\gamma^u(j_1-j_2)+\gamma^u(i_1-j_1)\gamma^u(i_2-j_2)+\gamma^u(i_1-j_2)\gamma^u(i_2-j_1).
$$
The Cauchy-Schwarz inequality applied twice gives
$$
|EX_1(s_{i_1})X_1(s_{i_2})X_1(s_{j_1})X_1(s_{j_2})|\leq \left(EX_1^4(s_{i_1})EX_1^4(s_{i_2})EX_1^4(s_{j_1})EX_1^4(s_{j_2})\right)^{1/4}.
$$
By Assumption~\ref{a:signal}~(ii) this term is bounded by $C_X$. Combing these results with Assumption~\ref{a:noise} it is easy to derive that 
$$
E\max_{1\leq t\leq T}\left(\sum_{\ell=1\atop \ell\neq t}^T (u_\ell^\top X_t(\bs))^2\right)^2\leq C_X\times  C_u^2\times \big(T^3+2T^2)p^2.
$$
Thus
$$
\frac{\hat\gamma_L}{p}\max_{1\leq t\leq T}\|A_{2t}\|=O_P\Big(\sqrt{L}\frac{T^{1/4}}{\sqrt{p}} \Big).
$$
This yields \eqref{e:ak112}. Relation \eqref{e:ak113} can be derived analogously. 
\end{proof}

\begin{proof}[\textbf{ \upshape Proof of Lemma~\ref{l:bj}:}]
We have
\begin{align*}
\|B_{1j}\|=\left\|\frac{1}{T}\sum_{t=1}^T y_{tj}\left(\hat f_t^\top-H f_t^\top\right)\right\| \leq \left(
\frac{1}{T}
\sum_{t=1}^T y_{tj}^2\frac{1}{T}\sum_{t=1}^T \left\|\hat f_t^\top-H f_t^\top\right\|^2
\right)^{1/2},
\end{align*}
where $y_{tj}=X_t(s_j)+u_{tj}$.
Lemma~\ref{l:ak12} implies $T^{-1}\sum_{t=1}^T \big\|\hat f_t^\top-H f_t^\top\big\|^2=O_P\big(L^{2}(\hat\gamma_L/p)^{-2}(1/T+1/p)\big)$. Moreover, 
$$
\frac{1}{T}\sum_{t=1}^T y_{tj}^2\leq 2\sup_{s\in [0,1]} \frac{1}{T}\sum_{t=1}^T X_t^2(s)+2\max_{k=1,\ldots, p} \frac{1}{T}\sum_{t=1}^T u_{tk}^2.
$$
The required $L^4$-$m$-approximability implies that the sequence $(\sup_{s\in [0,1]}X_t^2(s)\colon t\geq 1)$ is ergodic and hence by the ergodic theorem and Assumption~\ref{a:signal} (iii)
$$
\sup_{s\in [0,1]} \frac{1}{T}\sum_{t=1}^T X_t^2(s)\leq  \frac{1}{T}\sum_{t=1}^T \sup_{s\in [0,1]} X_t^2(s)=O_P(1).
$$
Note that by Assumption~\ref{a:noise} the variables $u_{tj}^2/(\gamma^u(0))^2$ are $\chi^2$-distributed with 1 degree of freedom and are independent across $t$. Without loss of generality let us assume that $\gamma^u(0)=1$. Then
\begin{align*}
\Pr\left(\max_{k=1,\ldots, p} \frac{1}{T}\sum_{t=1}^T u_{tk}^2>3\right)&\leq p \Pr\left(e^{\frac{1}{3}\sum_{t=1}^T u_{tk}^2}>e^{T}\right) \leq p E\left(e^{\frac{1}{3}\sum_{t=1}^T u_{t1}^2}\right)e^{-T} \leq p \Big(Ee^{\frac{1}{3}u_{11}^2}\Big)^Te^{-T} \\
&=p\left(e/\sqrt{3}\right)^{-T}\leq p e^{-0.4T}.
\end{align*}
Since we require that $p\to\infty$ at a sub-exponential rate, this tends to zero. Thus we conclude that \eqref{e:bj1} holds.

As for \eqref{e:bj2}, we have
$$\max_{1 \leq j \leq p} \Vert \frac{1}{T} H F^\top U_j \Vert \leq \Vert H \Vert \max_{1 \leq j \leq p} \Vert \frac{1}{T} F^\top U_j \Vert.$$
Using Lemma~\ref{l:Hbound} and Lemma~\ref{l:Bj2}, the claim follows. 
For $B_{3j}$, we note that with Assumption \ref{a:signal} it follows that 
$$\max_{1 \leq j \leq p} \Vert b_j \Vert = \max_{1 \leq j \leq p} \Big( \sum_{\ell=1}^L \lambda_\ell \varphi_\ell^2(s_j)\Big)^{1/2} = \max_{1 \leq j \leq p} \sqrt{EX^2_t(s_j)} = O(1).$$
The claim then follows directly from Lemma~\ref{l:Hbound} and Lemma~\ref{l:FFI}.
This finishes the proof.
\end{proof}

\begin{proof}[\textbf{ \upshape Proof of Lemma~\ref{l:bs}:}]
We just have to note that $\|b_j\|^2=\mathrm{Var}(X_1(s_j))$.
\end{proof}

\begin{proof}[\textbf{ \upshape Proof of Lemma~\ref{l:Hbound}:}]
It holds that 
$$
\|H\|\leq \|(\hat \Lambda/p)^{-1}\| \|\hat F/\sqrt{T}\|\|F/\sqrt{T}\|\|B^\top B/p\|.
$$
Thus, the following easily verifiable assertions imply the proof:
\begin{align}
 \|(\hat \Lambda/p)^{-1}\|&=p/\hat\gamma_L; \qquad \|\hat F/\sqrt{T}\| \leq  \|\hat F/\sqrt{T}\|_F=\sqrt{L};\nonumber\\
 \|F/\sqrt{T}\|^2&\leq \|F/\sqrt{T}\|_F^2=\left(\sum_{\ell=1}^L \frac{1}{T}\sum_{t=1}^T\frac{1}{\lambda_\ell}\langle X_t,\varphi_\ell\rangle^2\right) =\sum_{\ell=1}^L\frac{1}{\lambda_\ell}\bigg\langle\Big(\frac{1}{T}\sum_{t=1}^T X_t\otimes X_t\Big)(\varphi_\ell),\varphi_\ell\bigg\rangle\nonumber\\
 &\leq \frac{L}{\lambda_L}\Big\|\frac{1}{T}\sum_{t=1}^T X_t\otimes X_t\Big\|=O_P(L/\lambda_L);\label{e:H1}\\
\|B^\top B/p\|^2&\leq \|B^\top B/p\|_F^2=\sum_{k=1}^L \sum_{\ell=1}^L \lambda_k\lambda_\ell\left(\frac{1}{p}\sum_{i=1}^p\varphi_k(s_i)\varphi_\ell(s_i)\right)^2 = O(1).\label{e:H2}
\end{align}
In \eqref{e:H1} we used Lemma~\ref{l:XtimesX}, while \eqref{e:H2}  follows from $\sum_{k\geq 1}\lambda_k<\infty$ and Assumption~\ref{a:pcs}.
\end{proof}

\begin{proof}[\textbf{ \upshape Proof of Lemma~\ref{l:H}:}]
Let us first derive a bound for $\|HH^\top-I_L\|$.
We have 
\begin{align}
\|HH^\top-I_L\|&\leq \|HH^\top -\frac{1}{T}HF^\top F H^\top\|+\|\frac{1}{T}HF^\top F H^\top-I_L\|\nonumber\\
&\leq \|H\|^2\|I_L -\frac{1}{T}F^\top F \|_F+\|\frac{1}{T}H F^\top F H^\top-I_L\|.\label{e:HH}
\end{align}
For the first term in \eqref{e:HH} by Lemmas~\ref{l:Hbound} and \ref{l:FFI}

$$\|HH^\top -\frac{1}{T}HF^\top F H^\top\| = O_P\left(\frac{L}{\sqrt{T}} \left(\frac{L}{\lambda_L}\frac{p}{\hat\gamma_L}\right)^2 \right)$$ 
 and the second term is equal to 
\begin{align*}
\|\frac{1}{T}H F^\top F H^\top-\frac{1}{T}\hat F^\top \hat F\|\leq \frac{1}{T}\|H F^\top- \hat F^\top \|_F\|H\|\|F\|_F+\frac{1}{T}\|\hat F\|_F\|\hat F-FH^\top\|_F.
\end{align*}
We have $\|H\|=O_P\Big(pL/\hat \gamma_L \sqrt{\lambda_L}\Big)$ (Lemma~\ref{l:Hbound}) and $\|F\|_F=O_P\Big(\big(T L/\lambda_L\big)^{1/2}\Big)$  (see proof of Lemma~\ref{l:Hbound}) and $\|\hat F\|_F=\sqrt{LT}$. From Lemma~\ref{l:ak12} we deduce that
$$
\frac{1}{T}\|H F^\top- \hat F^\top \|_F= O_P\left(\frac{1}{\sqrt{T}}\Big(L^{2}(\hat\gamma_L/p)^{-2}\big(1/T+1/p\big)\Big)^{1/2}\right).
$$
Combined, this shows that the second term in \eqref{e:HH} is
$$
O_P\left(
\left(\frac{p}{\hat{\gamma}_L}\right)^2\frac{L^{5/2}}{\lambda_L}
\left(
\frac{1}{\sqrt{T}}+\frac{1}{\sqrt{p}}
\right)
\right).
$$
Thus 
\begin{equation}
\|H H^\top-I_L\| = O_P\left(\left(\frac{p}{\hat\gamma_L}\right)^2\frac{L^3}{\lambda_L^2}\left(
\frac{1}{\sqrt{T}}+\frac{1}{\sqrt{p}}
\right)\right).\label{e:HH1}
\end{equation}
Define $D_L:=H H^\top-I_L$ and let us denote by $S=S_T=\{\omega\in\Omega\colon \|D_L\|<1\}$.  If the bound in \eqref{e:HH1} tends to zero, then $\Pr(S)\to 1$. 
On $S$ we have that
$
W_L:=I_L-D_L+D_L^2-D_L^3\pm \ldots
$ converges and it is easily seen that $W_L$ is the inverse of $I_L+D_L=H H^\top$. But then $H$ is invertible and $H^{-1}=H^\top W_L$. 

Hence we get
\begin{equation}\label{e:Hinv}
\Pr(S \cap \{\|H^{-1}\|\leq 2\|H\|\})\geq \Pr(S\cap \|W_L\|\leq 2)\geq \Pr(\|D_L\|\leq 1/2)\to 1.
\end{equation}
Then, using  $H^\top H-I_L=H^{-1}D_LH$, we deduce that on $S$ 
$$
\|H^\top H-I_L\|\leq \|D_L\|\|H^{-1}\|\|H\|.
$$
So the statement of the lemma follows from \eqref{e:HH1}, Lemma~\ref{l:Hbound} and \eqref{e:Hinv}.
\end{proof}

\section{Weakening the assumption on Gaussian noise}
\label{appendix:gauss}

As mentioned in Section~\ref{s:assumptions}, one may weaken the assumed Gaussianity of the error process in Assumption~\ref{a:noise}. We discuss a rough outline of the consequences of avoiding this specific restriction. In the following we instead assume 8 moments for $u_t$ and the validity of some moment inequalities which are well known in the iid case and which have been verified under Assumption~\ref{a:noise}. Additionally, we assume $p=O(T^2)$ in order to guarantee the same convergence rates in our theorems. Previously, subexponential growth of $p=p(T)$ was possible. Below we discuss the necessary modifications. 
\begin{enumerate}[label=(\Roman*)]
\item In Lemma~\ref{l:noisemoments} we prove (i) $E(u_2^\top u_1)^{2k}\leq c_2 p^k$ and (ii) $E\|u_1\|^{4k}\leq c_2 p^{2k}$, $k \in \lbrace 1,2 \rbrace$. In the case of iid noise this follows immediately from Rosenthal's inequality (see \cite{petrov:1975}). To apply this inequality we need for (i) $2k$ moments, and for (ii) $4k$ moments, which ultimately means that we need 8 moments. Extensions to stationary processes exist (see e.g.\ \cite{liu:xiao:wu:2013}). 
In order to provide a more general setting, we now state (i) and (ii) as assumptions.
\item \label{itm:app2} In Lemma~\ref{l:Bj2} we need a bound for
\begin{equation}\label{e:l13}
\sum_{j\leq p \atop i\leq L}\Pr\left( 
\Big|
\frac{1}{T}\sum_{t=1}^T\frac{1}{\sqrt{\lambda_i}}\langle X_t,\varphi_i\rangle u_{tj}\Big|>\kappa
\cap \max_{k\leq L} \frac{1}{T}\sum_{t=1}^T\frac{1}{ \lambda_k}\langle X_t,\varphi_k\rangle^2\leq 2
\right).
\end{equation}
We will use the following version of the Markov inequality:
\begin{lemma}
Suppose that $(A_t)$ and $(V_t)$ are random sequences which are independent of each other. Assume that $(V_t)$ is iid zero mean and has a finite fourth moment.  Then
$$
p_1(\kappa)=\Pr \left(\Big|\frac{1}{T}\sum_{t=1}^T A_tV_t\Big|>\kappa \cap \frac{1}{T}\sum_{t=1}^T A_t^2\leq 2\right)\leq\frac{\mathrm{16 EV_1^4}}{\kappa^4 T^2}.
$$
\end{lemma}
\begin{proof}[\textbf{\upshape Proof:}]
We have
\begin{align*}
&\Pr\left(\Big|\frac{1}{T}\sum_{t=1}^T A_tV_t\Big|>\kappa \cap \frac{1}{T}\sum_{t=1}^T A_t^2\leq 2\right) \leq \frac{1}{T^4\kappa^4} E\left(\Big|\sum_{t=1}^T A_tV_t\Big|^4I\{\frac{1}{T}\sum_{t=1}^T A_t^2\leq 2\}\right)\\
&\quad = \frac{1}{T^4\kappa^4} E\left(\sum_{r=1}^T \sum_{s=1}^T\sum_{t=1}^T \sum_{u=1}^TA_rA_sA_t A_uV_rV_sV_t V_u I\{\frac{1}{T}\sum_{t=1}^T A_t^2\leq 2\}\right)\\
&\quad = \frac{1}{T^4\kappa^4} E\left(
\sum_{r=1}^T \sum_{s=1}^T\sum_{t=1}^T \sum_{u=1}^T E\left[A_rA_sA_t A_uV_rV_sV_t V_u I\{\frac{1}{T}\sum_{t=1}^T A_t^2\leq 2\}\bigg| A_1,\ldots, A_T\right]\right)\\
&\quad = \frac{1}{T^4\kappa^4} E\left(\sum_{r=1}^T \sum_{s=1}^T\sum_{t=1}^T \sum_{u=1}^T E(V_rV_sV_t V_u)E\left[A_rA_sA_tA_u I\{\frac{1}{T}\sum_{t=1}^T A_t^2\leq 2\}\bigg| A_1,\ldots, A_T\right]\right)\\
&\quad \leq EV_1^4 \frac{1}{T^4\kappa^4} E\left(\sum_{t=1}^T A_t^4 I\{\frac{1}{T}\sum_{t=1}^T A_t^2\leq 2\}\right) +3(EV_1^2)^2\frac{1}{T^4\kappa^4}E\left(\sum_{s=1}^T A_s^2 \sum_{t=1}^T A_t^2I\{\frac{1}{T}\sum_{t=1}^T A_t^2\leq 2\}\right).
\end{align*}
Using $(EV_1^2)^2\leq EV_1^4$ and $\sum_{t=1}^T A_t^4\leq \left(\sum_{t=1}^T A_t^2\right)^2$, the result is immediate.
\end{proof}
The lemma is now applied to \eqref{e:l13} with $\kappa=c (Lp)^{1/4}T^{-1/2}$. Then for any $\epsilon>0$ there is a $c$ independent of $T$, $p$ and $L$ such that $p_1(\kappa)<\epsilon$. Using this in Lemma~\ref{l:Bj2} implies that 
$$
\max_{j\leq p}\left\|\frac{1}{T} F^\top U_j\right\|=O_P\left(\frac{(Lp)^{1/4}}{T^{1/2}}\right).
$$
If again we assume 8 moments for $u_{t}$, the analogue approach yields
$$
\max_{j\leq p}\left\|\frac{1}{T} F^\top U_j\right\|=O_P\left(\frac{(Lp)^{1/8}}{T^{1/2}}\right).
$$
\item 
The bound in \ref{itm:app2}\ is used in the proof of Lemma~\ref{l:bj}. In \eqref{e:bj2} we obtain the modification
\begin{equation}\label{e:B2}
\frac{\hat\gamma_L}{p}\max_{j=1}^p\|B_{2j}\|=O_P\left(\frac{L^{5/8}p^{1/8}}{\sqrt{\lambda_L}\sqrt{T}}\right).
\end{equation}
This implies that the bound for $R^{(2)}$ needs to be modified. In the proof of Theorem~\ref{thm} (where $L$ is fixed) this needs to be replaced by
$$
O_P\left(1/T^{1/4}+T^{1/4}/\sqrt{p}+p^{1/8}/\sqrt{T}\right).
$$
Since we require $p=O(T^2)$ the rate remains unchanged. Moreover, the factor $L$ in \eqref{e:B2} is included with a smaller power than in \eqref{e:bj2}, and hence this modification also does not affect the rate in Theorem~\ref{thm2}, where $L$ is allowed to grow.
\item In the proof of Lemma~\ref{l:ak11} we need 
$$
E\left(\sum_{j=1}^p u_{1j}\right)^4=O(p^2).
$$
In the iid case this can again be deduced by the Rosenthal's inequality. Likewise we may use some extension thereof in the dependent case.
\item In the proof of Lemma~\ref{l:bj} we  bound
\begin{align*}
\Pr\left(\max_{k=1,\ldots, p}\frac{1}{T} \sum_{t=1}^T u_{tk}^2>2Eu_{11}^2\right)&\leq
p \Pr\left(\sum_{t=1}^T (u_{t1}^2-Eu_{11}^2)>Eu_{11}^2 T\right)=O\left(\frac{p}{T^2}\right),
\end{align*}
by the Rosenthal inequality. Note that with respect to $t$ the $u_{t}$ are assumed to be independent. For our proof this term is required to be $O(1)$, which is the case due to the assumption of $p=O(T^2)$.
\end{enumerate}

\section*{Acknowledgements} We would like to thank the editors and the anonymous reviewers for their interesting questions and very constructive remarks which helped improve the paper significantly.

\bibliographystyle{abbrv}
\bibliography{paper}

\end{document}